\documentclass[a4paper,10pt]{amsart}
\usepackage[T1]{fontenc} 
\usepackage[utf8]{inputenc}
\usepackage{lmodern}
\usepackage{enumerate}
\usepackage{epsfig} \input{xy} \xyoption{all}
\usepackage{amsfonts,amsmath,amstext,amsbsy,amssymb}
\usepackage{amsbsy,amsopn,amsthm,amscd,amsxtra}
\usepackage{latexsym,mathrsfs}
\usepackage{mathabx}

\usepackage{hyperref}
\RequirePackage[dvipsnames]{xcolor} 
\definecolor{halfgray}{gray}{0.55} 
\definecolor{webgreen}{rgb}{0,0.5,0}
\definecolor{webbrown}{rgb}{.6,0,0} \hypersetup{%
  colorlinks=true, linktocpage=true, pdfstartpage=3,
  pdfstartview=FitV,%
  breaklinks=true, pdfpagemode=UseNone, pageanchor=true,
  pdfpagemode=UseOutlines,%
  plainpages=false, bookmarksnumbered, bookmarksopen=true,
  bookmarksopenlevel=1,%
  hypertexnames=true,
  pdfhighlight=/O,
  urlcolor=webbrown, linkcolor=RoyalBlue,
  citecolor=webgreen, 
  pdftitle={},%
  pdfauthor={},%
  pdfsubject={2000 MAthematical Subject Classification: Primary:},%
  pdfkeywords={},%
  pdfcreator={pdfLaTeX},%
  pdfproducer={LaTeX with hyperref}%
}

\newcommand{\abs}[1]{\left\lvert{#1}\right\rvert}
\newcommand{\norm}[1]{\left\|{#1}\right\|}

\newcommand{\mc}{\mathcal} 

\newcommand{\R}{\mathbb{R}}\newcommand{\N}{\mathbb{N}}
\newcommand{\Z}{\mathbb{Z}}

 \newcommand{\ie}{i.e.\ }
\newcommand{\eg}{e.g.\ }

\newtheorem*{thm*}{Theorem}

\newtheorem{theorem}{Theorem}[section]
\newtheorem{proposition}[theorem]{Proposition}
\newtheorem{lemma}[theorem]{Lemma}
\newtheorem{corollary}[theorem]{Corollary}

\theoremstyle{definition} \newtheorem{definition}[theorem]{Definition}

\theoremstyle{remark}

\newcommand{\Diff}[2]{\mathrm{Diff}_{#1}^{#2}}

 \newcommand{\carr}{\righttoleftarrow}

 \DeclareMathOperator{\Fix}{Fix}
 \DeclareMathOperator{\GL}{GL}

\title[Liv\v{s}ic theorem]{Liv\v{s}ic theorem for diffeomorphism
  cocycles}

\author[A. Avila]{Artur Avila} 
\address[Avila]{CNRS, IMJ-PRG,  UMR 7586, Univ  Paris  Diderot,  Sorbonne  Paris  Cite,  Sorbonnes Universit\'es, UPMC Univ Paris 06, F-75013, Paris, France $\&$ IMPA, Estrada Dona Castorina 110, Rio de Janeiro, Brasil} 
\email{artur@math.univ-paris-diderot.fr}

\author[A. Kocsard]{Alejandro Kocsard} \address[Kocsard]{Instituto de
  Matem\'{a}tica e Estat\'{\i}stica - Universidade Federal
  Fluminense. Rua Professor Marcos Waldemar de Freitas Reis, s/n. Bloco H - Campus de Gragoat\'{a}.
24.210-201, Niter\'{o}i, RJ - Brasil}
\email{akocsard@id.uff.br}

\author[X. Liu]{Xiao-Chuan Liu}\address[Liu]{IMPA, Estrada Dona Castorina 110, Rio de Janeiro, Brasil} 
\email{lxc1984@impa.br}

\date{\today}

\begin{document}

\maketitle

\begin{abstract}
  We prove the so called Liv\v{s}ic theorem for cocycles taking values
  in the group of $C^{1+\beta}$-diffeomorphisms of any closed manifold
  of arbitrary dimension. Since no localization hypothesis is assumed,
  this result is completely global in the space of cocycles and thus
  extends the previous result of the second author and
  Potrie~\cite{KocsardPotrieLivThm} to higher dimensions.
\end{abstract}

\section{Introduction}
\label{sec:introduction}

Given a dynamical system $f\colon M\to M$ and a map $A\colon M\to G$,
where $G$ denotes a topological group, a fundamental problem in the
modern theory of dynamical systems consists in determining whether
there exists a function $u\colon M\to G$ such that
\begin{displaymath}
  A(x)=u\big(f(x)\big)\big(u(x)\big)^{-1}, \quad\forall x\in M.
\end{displaymath}
Such an equation is usually called a \emph{cohomological equation} and
when it does admit a solution, the cocycle $A$ is said to be a
\emph{coboundary} for $f$.

A celebrated theorem of Liv\v{s}ic \cite{LivCertPropHomolYSyst,
  LivsicCohomDynSys, LivHomolDynSys} claims that when $f$ is a
hyperbolic homeomorphism, the group $G$ is the additive real line $\R$
and $A$ is H\"older continuous, then there exists a continuous
solution $u\colon M\to\R$ satisfying
\begin{equation}
  \label{cohoequation}
  A(x)=u \circ f-f,
\end{equation}
if and only if
\begin{equation}
  \label{periodiccondition}
  \sum_{i=0}^{n-1}A(f^i(p))=0, \quad\forall p\in\Fix(f^n),\ \forall
  n\geq 1.
\end{equation}

Many generalizations have been studied. Some of them consider more
general dynamical systems on the base. For example, in
\cite{KatokKononenkoCocStab} and \cite{WilkinsonCohomoEq}
cohomological equations over partially hyperbolic diffeomorphisms are
studied. Other generalizations deal with cocycles taking values in
more general groups (see for instance
\cite{delaLlaveWindsorLivThmNonComm, delaLlaveWindsorLivThmNonComm}).

In this paper, we consider the second kind of generalization. When $G$
is an arbitrary group, the above condition in Liv\v{s}ic theorem
becomes
\begin{displaymath}
  A(f^{n-1}(p))\cdots A(p)=e_G,
\end{displaymath}
for any periodic point $p$ of period $n$, and where $e_G$ denotes the
identity element of the group $G$.

Liv\v{s}ic himself has already considered more general groups in
\cite{LivCertPropHomolYSyst}. In fact, in that work he deals with
cocycles taking values in a topological groups admitting a complete
bi-invariant distance (\eg Abelian or compact groups), and in this
case, the very same techniques that were applied to the real case (\ie
when $G=\R$) still work.

Looking for generalizations of these results to cocycles taking values
in groups not admitting bi-invariant distances, many authors (for
instance \cite{LivsicCohomDynSys, 
delaLlaveWindsorLivThmNonComm,
  delaLlaveWindsorSmoorhDependence}) have applied ``distortion control
techniques'' which consist in endowing $G$ with a left-invariant
distance and to estimate the distortion by right
multiplication. However, in most of the cases this method alone only
works under certain ``localization hypothesis'', \ie assuming that the
cocycle is sufficiently close to the constant identity cocycle
in certain topology, and therefore, does not produce global results.

The first global Liv\v{s}ic like theorem for topological groups not
admitting a bi-invariant distance is due to
Kalinin~\cite{KalininLivThmMat}, who proved Liv\v{s}ic theorem for
linear cocycles substituting the localization hypothesis by some
estimates of Lyapunov exponents.

In the infinite-dimensional case, that is when $G$ is a topological
group of infinite dimension, Ni\c{t}ic\u{a} and
T\"or\"ok~\cite{NiticaTorokCohomolDynSyst,
  NiticaTorokRegularResultsSolLiv} got the first local results (\ie
under some localization hypothesis) for groups of diffeomorphisms of
closed manifolds. Some extensions of these results were obtained by de
la Llave and Windsor in \cite{delaLlaveWindsorLivThmNonComm}, but
always under certain localization hypotheses. Navas and Ponce dealt in
\cite{NavasPonceLivAnalGerms} with the case of cocycles taking values
in the group of analytic germs at the origin.

The first global Liv\v{s}ic type theorem for groups of diffeomorphisms
was recently obtained by Potrie and the second author of this paper in
\cite{KocsardPotrieLivThm}. There they got a sharp characterization of
coboundaries in terms of Lyapunov exponents, introducing some new
geometric arguments and avoiding the usage of control distortion
techniques. With this characterization of coboundaries, that holds for
groups of diffeomorphisms of manifolds of any dimension, they proved a
Liv\v{s}ic type theorem for cocycles with values in the group of
circle $C^1$-diffeomorphisms.

The main result of this paper is the so called Liv\v{s}ic theorem for
H\"older cocycles taking values in the group of $C^r$-diffeomorphisms
of a closed manifold, with $r>1$, extending the previous result of
\cite{KocsardPotrieLivThm} to higher dimensions.

\subsection{Main results}
\label{sec:main-results}

All along this work, $M$ will denote a compact metric space,
$f\colon M\carr$ a hyperbolic homeomorphism, $N$ a closed smooth
manifold of dimension $q$ and $\Diff{}r(N)$ the group of
$C^r$-diffeomorphisms of $N$ (see \S~\ref{sec:prelim-notations} for
details and definitions).

The main result of this work is the following

\begin{theorem}
  \label{thm:main-thm}
  Let $A\colon M\to\Diff{}{1+\beta}(N)$ be a H\"older continuous
  cocycle, with $\beta>0$. Then, $A$ satisfies the \emph{periodic
    orbit condition,} \ie,
  \begin{equation}
    \label{eq:POC}
    A(f^{n-1}(p))\cdots A(f(p))A(p)=id_N, \quad\forall p\in\Fix(f^n),\  
    \forall n\geq 1,
  \end{equation}
  if and only if there exists a H\"older continuous map
  $u:M\to \Diff{}{1+\beta}(N)$ such that
  \begin{displaymath}
    A(x)=u\big(f(x)\big)\circ u(x)^{-1}, \quad\forall x \in M.
  \end{displaymath}
\end{theorem}

As a consequence of Theorem~\ref{thm:main-thm} and some results due to
de la Llave and Windsor~\cite{delaLlaveWindsorSmoorhDependence}, we
get the Liv\v{s}ic theorem for groups of diffeomorphisms in higher
regularity:
\begin{corollary}
  \label{cor:livsic-higher-reg}
  Let $A\colon M\to\Diff{}{r}(N)$ be a H\"older continuous cocycle,
  with $r>1$. Then, $A$ satisfies the periodic orbit condition
  \eqref{eq:POC} if and only if there exists a H\"older continuous map
  $u\colon M\to\Diff{}{r}(N)$ such that
  \begin{displaymath}
    A(x)=u(f(x))\circ u(x)^{-1}, \quad\forall x\in M.
  \end{displaymath}
\end{corollary}

\subsection*{Acknowledgments} While this work was in preparation, we
learned that M. Guysinsky was working on the very same problem, but
applying rather different techniques. We would like to thank him for
kindly sharing a draft of his work \cite{GuysinskyLivsic} with us
while A.K. and X.L. were visiting Penn State University. A.K. is very
indebted to Federico Rodríguez-Hertz for several insightful
discussions and his hospitality during that visit.  

A.K. was partially supported by CNPq-Brazil (Bolsa de Produtividade em
Pesquisa) and FAPERJ-Brazil (Jovem Cientista do Nosso Estado).
X.L. was supported by CNPq-Brazil (Bolsa de P\'os-Doutorado J\'unior).

\section{Preliminaries and notations}
\label{sec:prelim-notations}

\subsection{$C^{1+\beta}$-distance on groups of diffeomorphisms}
\label{differentialtopology}

All along this article $N$ will denote a compact Riemannian manifold
and we shall write $d_N$ for its distance function induced by the
Riemannian structure. 

We shall consider the group of $C^{1+\beta}$-diffeomorphisms of $N$,
which is denoted by $\Diff{}{1+\beta}(N)$, endowed with the distance
function $d_{C^{1+\beta}}$ which is defined as follows: first we
define the $d_{C^0}$ distance by
\begin{displaymath}
  d_{C^0}(f,g):=\max\left\{d\big(f(x),g(x)\big) : x\in N\right\},
  \quad\forall f,g\in\Diff{}{1+\beta}(N).
\end{displaymath}

Then, since $N$ is compact, there is finite atlas
$\{\phi_i\colon U_i\subset N\to\R^q,\ i=1,\ldots,m\}$ and compact sets
$K_i\subset U_i$ such that $\bigcup_{i=1}^n K_i=N$. So, given any pair
of diffeomorphisms $f,g\in\Diff{}{1+\beta}(N)$ satisfying
\begin{equation}
  \label{eq:f-g-near}
  f\circ
  g^{-1}(K_i),  g\circ f^{-1}(K_i)\subset U_i, \quad\forall
  i\in\{1,\ldots,m\}, 
\end{equation}
we define the distance $d_{C^1}$ by
\begin{displaymath}
  \begin{split}
    d_{C^1}(f,g):=d_{C^0}(f,g) + \max_{1\leq i\leq m}\left\{\max_{x\in
        \phi_i(K_i)} \norm{D(\phi_i\circ f\circ
        g^{-1}\circ\phi_i^{-1}(x))} \right\} \\
    + \max_{1\leq i\leq m}\left\{\max_{x\in
        \phi_i(K_i)} \norm{D(\phi_i \circ g\circ
        f^{-1}\circ\phi_i^{-1}(x))} \right\}.
  \end{split} 
\end{displaymath}

Then, we define the distance $d_{C^{1+\beta}}$ by
\begin{displaymath}
  \begin{split}
    &d_{C^{1+\beta}}(f,g):= d_{C^1}(f,g) + \\
    &\max_{1\leq i\leq m}\left\{\max_{x,y\in\phi_i(K_i), x\neq
        y}\frac{\norm{D(\phi_i\circ f\circ g^{-1}\circ\phi_i^{-1}(x)) -
          D(\phi_i\circ f\circ
          g^{-1}\circ\phi_i^{-1}(y)) }}{\norm{x-y}^\beta} \right\} +\\
    &\max_{1\leq i\leq m}\left\{\max_{x,y\in\phi_i(K_i), x\neq
        y}\frac{\norm{D(\phi_i\circ g\circ f^{-1}\circ\phi_i^{-1}(x)) -
          D(\phi_i\circ g\circ f^{-1}\circ\phi_i^{-1}(y))
        }}{\norm{x-y}^\beta} \right\}.
  \end{split}
\end{displaymath}

Finally, we define $d_{C^1}(f,g)= d_{C^0}(f,g) + 1$ and
$d_{C^{1+\beta}}(f,g) = d_1{C^1}(f,g) +1$, whenever condition
\eqref{eq:f-g-near} does not hold.



\subsection{Hyperbolic homeomorphisms}
\label{sec:hyperb-home}

Given a compact metric space $(M,d)$, a homeomorphism
$f\colon M\carr$ and an arbitrary point $x\in M$, let us consider the
\emph{local stable and unstable sets at $x$} given by
\begin{align*}
  W^s_\varepsilon(x)&:=\left\{y\in M:d\big(f^n(x),f^n(y)\big)
                      \leq\varepsilon,\ \forall n\geq 0\right\}; \\ 
  W^u_\varepsilon(x)&:=\left\{y\in M:d\big(f^{-n}(x),f^{-n}(y)\big) 
                      \leq\varepsilon,\ \forall n\geq 0\right\}, 
\end{align*}
where $\varepsilon$ is any positive real number.

\begin{definition}
  \label{def:hyperbolic-homeo}
  The homeomorphism $f$ is said to be \emph{hyperbolic} whenever it is
  bi-Lipschitz (\ie $f$ and $f^{-1}$ are Lipschitz), transitive and
  there exist constants $\varepsilon,\delta,K_0,\tau>0$ and continuous
  real functions $\nu_s,\nu_u\colon M\to (0,\infty)$ such that the
  following properties hold:
  \begin{enumerate}[(i)]
  \item $d(f(y'),f(y''))\leq \nu_s(y) d(y',y'')$ for any
    $y',y''\in W_\varepsilon^s(y)$ and any $y\in M$;
  \item $d(f^{-1}(y'),f^{-1}(y'')) \leq \nu_u(y) d(y',y'')$ for any
    $y',y''\in W_{\varepsilon}^u(y)$ and any $y\in M$;
  \item
    $\nu_s^{(n)} (y)=\nu_s(f^{n-1}(y))\cdots\nu_s(y) \leq K_0 e^{-\tau
      n}$ for any $y\in M$ and any $n\geq 1$;
  \item
    $\nu_u^{(n)}(y) =\nu_u (f^{-(n-1)}(y)) \cdots \nu_u(y) \leq K_0
    e^{-\tau n}$ for any $y\in M$ and any $n\geq 1$.
  \item if $d(y,y')<\delta$, then
    $W^u_\varepsilon(y)\cap W^s_\varepsilon(y')$ consists of exactly
    one point, denoted by $[y,y']$, which depends continuously on
    $(y,y')$,
  \end{enumerate}
\end{definition}

A fundamental property about hyperbolic homeomorphisms is given by the
so called Anosov closing lemma (see for instance \cite[Theorem
6.4.15]{KatokHasselblatt}):

\begin{lemma}
  \label{Anosovclosing}
  If $f\colon M\carr$ is a hyperbolic homeomorphism, then there exist
  constants $C,\delta>0$ such that for any $x\in M$ and any $n\geq 1$
  such that $d(x,f^n(x))<\delta$, there exists a unique periodic point
  $p\in\Fix(f^n)$ satisfying
  \begin{displaymath}
    d(f^i(x),f^i(p))\leq C d\big(x,f^n(x)\big)e^{-\tau \min(i,n-i)}
    \text{ for any } i=0,\cdots,n, 
  \end{displaymath}
  where $\tau$ is the constant given by
  Definition~\ref{def:hyperbolic-homeo}.
  
  Moreover, if we define $y:=[p,x]$, it holds
  \begin{displaymath}
    d(f^i(x),f^i(y))\leq C d\big(x,f^n(x)\big) e^{-\tau i},
  \end{displaymath}
  and
  \begin{displaymath}
    d(f^i(p),f^i(y))<C d\big(x,f^n(x)\big) e^{-\tau (n-i)},
  \end{displaymath}
  for any $i=0,1,\ldots,n$.
\end{lemma}

\subsection{Cocycles and coboundaries}
\label{sec:cocycles-coboundaries}

Let $G$ be a topological group and $f\colon M\carr$ be a hyperbolic
homeomorphism.  Every continuous map $A\colon M \to G$ shall be
considered as a \emph{$G$-cocycle} over $f$, \ie we will use the
following notation: given any $n\in\Z$, we write
\begin{displaymath}
  A^n(x):=
  \begin{cases}
    e_G,& \text{if } n=0;\\
    A\big(f^{n-1}(x)\big)\cdots A(x),& \text{if }
    n\geq 1, \\
    A\big(f^{n}(x)\big)^{-1} \cdots A\big(f^{-1}(x)\big)^{-1},&
    \text{if } n<0.
  \end{cases}
\end{displaymath}

It is important to notice that any $G$-cocycle over $f\colon M\carr$
naturally induces a skew-product map $F=(f,A)\colon M\times G\carr$
given by
\begin{equation}
  \label{eq:skew-product-general-case}
  (f,A)(x,g):=\big(f(x),A(x)g\big),\quad\forall (x,g)\in X\times G. 
\end{equation}
Observe that in this case, it holds $(f,A)^n(x,g)=(f^n(x),A^n(x)g)$,
for every $n\in\Z$, and any $(x,g)\in M\times G$.

\begin{definition}
  A $G$-cocycle $A$ is called a \emph{coboundary} whenever there
  exists a continuous mapping $u\colon M\to G$ satisfying
  \begin{equation}
    \label{eq:cohomo-equation} 
    A(x)=u\big(f(x)\big)\cdot u(x)^{-1},\quad\forall x\in M.
  \end{equation}
\end{definition}

An obvious, yet extremely important, family of obstructions for a
$G$-cocycle to be a coboundary is given by periodic orbits of $f$.  In
fact, if $A$ is a coboundary, it necessarily holds
\begin{equation}
  \label{eq:POO}
  A^n(p)=e_G, \quad\forall n\in\N,\ \forall p\in\Fix(f^n). 
\end{equation}

In this paper we will say that a $G$-cocycle $A$ satisfies the
\emph{periodic orbit condition} (\emph{POC} for short) whenever
property \eqref{eq:POO} holds.

\subsection{Fibered Lyapunov exponents}
\label{sec:fiber-lyap-exp}

In this work we are mainly interested in the study of $G$-cocycles
with $G<\Diff{}r(N)$, where $N$ is a closed smooth manifold.

From now on and till the end of the paper we will suppose that $N$
endowed with a Riemannian structure $\langle\cdot,\cdot\rangle$, and
the norm induced by the Riemannian structure will be denoted by
$\norm{\cdot}$.

In this case, when $G$ is a subgroup of the $\Diff{}r(N)$, we can
slightly modify the skew-product given by
\eqref{eq:skew-product-general-case} defining
$F=(f,A)\colon M\times N\carr$ by
\begin{equation}
  \label{eq:skew-product-diff-group}
  F(x,y):=\big(f(x),A(x)y\big), \quad\forall (x,y)\in M\times N.  
\end{equation}

Observe that the map $F$ is $C^r$ along the vertical fibers. So, we
can consider the ``partial derivative of $F$ respect to the $N$
coordinate'', defining the linear cocycle over $F$ given by
\begin{equation}
  \label{eq:derivative-F-cocycle}
  \partial F^n_{(x,y)}v:= D\big(A^n(x)\big)(y)(v), \quad\forall x\in M,\
  \forall y \in N,\ \forall v\in T_yN,
\end{equation}
and any $n\in\Z$.

In this way we can apply Oseledets theorem to guarantee that, for
every ergodic $F$-invariant probability measure $\mu$, there exist
$k\leq\dim N$, real numbers $\lambda_1'>\lambda_2'>\ldots>\lambda_k'$
and a measurable $F$-invariant set $\Lambda\subset M\times N$ with
$\mu(\Lambda)=1$, such that for every $z=(x,y)\in\Lambda$ there exists
a linear splitting $T_yN = E^1_z\oplus E_z^2\oplus\ldots E_z^k$ which
measurably depends on $z$, satisfies $\partial F_z(E_z^i)=E_{F(z)}^i$,
for every $z\in\Lambda$ and every $1\leq i\leq k$, and
\begin{equation}
  \label{eq:lyap-exp-definition}
  \lambda_i':=\lim_{n\to\pm\infty} \frac{1}{n} \log\left\|
    \partial F^n_z(v) \right\|,
  \quad\forall z\in\Lambda,\ \forall v\in
  E_z^i\setminus\{0\}, \ \forall i\leq k.
\end{equation}

\subsection{Lyapunov exponents and coboundaries}
\label{sec:lyap-expon-vs-cobound}

As it was already shown in \cite{KocsardPotrieLivThm}, when dealing
with cocycle over hyperbolic homeomorphisms with values in groups of
diffeomorphisms there is a very tight relation between Lyapunov
exponents and coboundaries. In fact, the main result of
\cite{KocsardPotrieLivThm} is the following:

\begin{theorem}[Theorem 3.1 in \cite{KocsardPotrieLivThm}]
  \label{thm:zero-Lyap-exp-vs-cobound}
  Let $f\colon M\carr$ be a hyperbolic homeomorphism and
  $A\colon M\to\Diff{}1(N)$ be a H\"older continuous cocycle over $f$
  satisfying the POC given by \eqref{eq:POO}.

  Then the following statements are equivalent:
  \begin{enumerate}
  \item\label{eq:sub-exp-growth} Every Lyapunov exponent given by
    \eqref{eq:lyap-exp-definition} with respect to any $F$-invariant
    measure vanish, \ie it holds
    \begin{displaymath}
      \lim_{n\to\pm\infty} \frac{1}{n}\log\norm{\partial F^n_z} = 0,
      \quad\forall z\in M\times N;
    \end{displaymath}
  \item The cocycle $A$ is a coboundary, \ie there exists a H\"older
    continuous map $u\colon M\to\Diff{}1(N)$ such that
    \begin{displaymath}
      A(x) = u\big(f(x)\big)\circ u(x)^{-1},\quad\forall x\in M. 
    \end{displaymath}
  \end{enumerate}
\end{theorem}

\noindent{\bf Remark.} We should note that, for our main statement, we have to consider $C^{1+\beta}$ topology instead of 
$C^1$ topology as in the above statement. The reason is that the proof involves techniques from Pesin theory. 

So, in order to show Theorem~\ref{thm:main-thm} it is enough to prove
the following

\begin{theorem}
  \label{thm:main-thm-Lyap-exp}
  If $f\colon M\carr$ is a hyperbolic homeomorphism and
  $A\colon M\to\Diff{}{1+\beta}(N)$ is a H\"older cocycle satisfying
  \emph{POC}, then
  \begin{displaymath}
    \lim_{n\to\pm\infty} \frac{1}{n}\log\norm{\partial F^n_z} = 0,
    \quad\forall z\in M\times N. 
  \end{displaymath}
\end{theorem}

It is worth mentioning that in \cite{KocsardPotrieLivThm} it is
considered a slightly more general setting, dealing with \emph{à
  priori} arbitrary fiber bundles. There it is shown that the fiber
bundle is \emph{à posteriori} trivial whenever the corresponding fiber
bundle map turns to be a coboundary (see \S~2.6 in
\cite{KocsardPotrieLivThm} for more details).

Since all our techniques are either local or semi-local, we could have
worked in the fiber bundle category, getting the very same
results. However, in order to avoid cumbersome unnecessary
technicalities we have decided to present all our work in product
spaces, \ie \emph{à priori} trivial fiber bundles.

\section{Some linear algebraic lemmas}
\label{sec:alg-lemmas}

Let us consider $\R^d$ endowed with its usual Euclidean structure. Let
us denote the closed ball of radius $r$ and centered at the origin by
$B_r:=\{x\in\R^d : \norm{x}\leq r\}$ and the sphere of radius $r$ by
$\mathbb{S}_r:=\{x\in\R^d : \norm{x}=r\}$.

Given an ellipsoid $E\subset\R^d$, we shall write
$r_1(E)\geq r_2(E)\geq \cdots \geq r_d(E)$ for the ordered lengths of
its semi-axes.

Our first result of this section is the following elementary
\begin{lemma}
  \label{lem:inclus-ellipsoides}
  If $E,F\subset\R^d$ are two ellipsoids centered at the origin and
  such that $E$ is contained in the interior of $F$, then it holds
  \begin{displaymath}
    r_i(E)\leq r_i(F),\quad\forall i\in\{1,\ldots,d\}.
  \end{displaymath}
\end{lemma}

\begin{proof}
  Let $\{e_1,e_2,\ldots,e_d\}$ and $\{f_1,f_2,\ldots,f_d\}$ be two
  orthonormal bases of $\R^d$ such that the directions of $e_i$ and
  $f_i$ coincide with the direction of the $i^{\mathrm{th}}$-axis of
  $E$ and $F$, respectively.
    
  First notice that, since $E$ is contained in the interior of $F$, it
  clearly holds $r_1(E)\leq r_1(F)$. Reasoning by contradiction, let
  us suppose there exists $j\in\{2,\ldots,d\}$ such that
  $r_i(E)\leq r_i(F)$, for every $i\leq j-1$ and $r_j(E)>r_j(F)$. Let
  $V_{E,j}:=\mathrm{span}\{e_1,\ldots,e_j\}$ and
  $V_{F}^j:=\mathrm{span}\{f_j,\ldots,f_d\}$. Notice that
  $\|v\|\geq r_j(E)$, for all $v\in E\cap V_{E,j}$; and
  $\|v\|\leq r_j(F)$, for every $v\in F\cap V_F^j$. However,
  $\dim V_{E,j} + \dim V_F^j = j + d-j+1 = d+1$. So, there exists some
  $v\in E\cap V_{E,j}\cap V_{F}^j$ and consequently,
  $\|v\|\geq r_j(E)>r_j(F)$, contradicting the fact that $E$ is
  contained in the interior of $F$.
\end{proof}

For any $A\in\GL(d,\R)$, the image of the unit sphere
$A(\mathbb{S}_1)$ is an ellipsoid centered at the origin, and the
\emph{singular values} of $A$ are defined by
\begin{displaymath}
  \sigma_i(A):=r_i\big(A(\mathbb{S}_1)\big), \quad\text{for }
  i\in\{1,\ldots,d\}.
\end{displaymath}

It is well-known that $\sigma_i(A)^2$ is an eigenvalue for $AA^t$,
where $A^t$ denotes the transpose of $A$.

Our first perturbative result is the following:

\begin{lemma}\label{approximating}
  \label{lem:cpt-nbh-pert-sing-values}
  Let $(A_n)_{n\geq 1}$ be a sequence of matrices in $\GL(d,\R)$ and
  suppose there exist real numbers
  $\lambda_1\geq\lambda_2\geq\cdots\geq\lambda_d$ and $\delta>0$ such
  that
  \begin{displaymath}
    \abs{\frac{1}{n}\log\big(\sigma_i(A_n)\big)-\lambda_i}\leq
    \frac{\delta}{2},\quad\forall n\geq 1,\ \forall
    i\in\{1,\ldots,d\}. 
  \end{displaymath}

  Then, given any compact set $K\subset\GL(d,\R)$, there exists
  $N=N(K)>0$ such that for every pair of sequences
  $(C_n)_{n\geq 1},(D_n)_{n\geq 1}\subset K$, it holds
  \begin{displaymath}
    \abs{\frac{1}{n}\log\big(\sigma_i(C_nA_nD_n)\big)-\lambda_i}\leq
    \delta, \quad\forall n\geq N,
  \end{displaymath}
  and every $i\in\{1,\ldots,d\}$.
\end{lemma}

\begin{proof}
  Since $K$ is compact, there exists $\ell>1$ such that
  \begin{displaymath}
    \max\{\norm{C},\norm{C^{-1}}\}\leq\ell, \quad\forall C\in K.  
  \end{displaymath}

  So, if $E$ is an arbitrary ellipsoid centered at the origin, then
  we have $\frac{1}{\ell}E\subset C(E)\subset \ell E$, and by
  Lemma~\ref{lem:inclus-ellipsoides} it holds
  \begin{displaymath}
    \frac{1}{\ell}r_i(E) \leq r_i\big(C(E)\big)\leq \ell r_i(E),
    \quad\forall C\in K,
  \end{displaymath}
  and every $i$.

  Consequently, for any $n\geq 1$ it clearly holds
  \begin{displaymath}
    \frac{1}{\ell^2}\sigma_i(A_n)\leq \sigma_i(C_nA_nD_n)\leq
    \ell^2\sigma_i(A_n),\quad\forall n\geq 1,
  \end{displaymath}
  and every $i$.

  So,
  \begin{displaymath}
    \begin{split}
      -\frac{2\log{\ell}}{n} + \frac{1}{n}\log\sigma_i(A_n) -\lambda_i
      &\leq \frac{1}{n}\log\big(\sigma_i(C_nA_nD_n)\big)-\lambda_i \\
      & \leq \frac{2\log{\ell}}{n} + \sigma_i(A_n)-\lambda_i,
    \end{split}
  \end{displaymath}
  for every $n$ and any $i$, and taking $N:=\frac{4\log\ell}{\delta}$
  the lemma is proved.
\end{proof}

\begin{lemma}\label{singularapproximating}
  \label{lem:pertur-each-iterate}
  Let $\R^d=\bigoplus_{i=1}^k E^i$ be an orthogonal linear
  decomposition and write $d_i:=\dim E^i$. Consider real numbers
  $\lambda_1'>\lambda_2'>\ldots>\lambda_k'$, define
  \begin{displaymath}
    \kappa:=\frac{1}{2}\min_{1\leq i\leq k-1}\{\lambda_i'-\lambda_{i+1}'\},
  \end{displaymath}
  and let $\delta$ be any number satisfying $0<\delta<\kappa/2$.

  Let $(A_n)_{n\geq 1}\subset\GL(d,\R)$ be a sequence of matrices such
  that
  \begin{equation}
    \label{eq:Ei-A-inv}
    A_n(E^i)=E^i,
  \end{equation}
  and
  \begin{equation}
    \label{eq:Ei-A-norm-estimate}
    \norm{A_n\big|_{E^i}},\norm{A_n^{-1}\big|_{E^i}}^{-1}\in
    \left(e^{\lambda_i'-\frac{\delta}{4}},
      e^{\lambda_i'+\frac{\delta}{4}}\right), \quad\forall n\geq 1,
  \end{equation}
  and every $i\in\{1,\ldots,k\}$.

  Then, there exist $\alpha_1=\alpha_1(\delta)>0$ so that for every
  $(C_n)_{n\geq 1}\subset\GL(d,\R)$ satisfying
  \begin{displaymath}
    \norm{A_n-C_n}<\alpha_1,\quad\forall n\geq 1,
  \end{displaymath}
  it holds
  \begin{equation}
    \label{eq:sigmai-An-vs-Cn}
    \abs{\frac{1}{n}\log\sigma_i\Big(A^{(n)}\Big)-
      \frac{1}{n}\log\sigma_i\Big(C^{(n)}\Big)} 
    <\delta, \quad\forall n\geq 1,\ \forall i\in\{1,\ldots,k\},
  \end{equation}
  where $A^{(n)}:=\prod_{j=1}^n A_j$ and $C^{(n)}:=\prod_{j=1}^nC_j$.
\end{lemma}

\begin{proof}
  Given any $v\in\R^d$, we write $v=\sum_{i=1}^k v^i$ for its unique
  decomposition with $v^i\in E^i$, for each $i$.

  For each $j\in\{1,\ldots,k-1\}$ and $\gamma>0$, let us consider the
  cones
  \begin{displaymath}
    K_{\gamma}^j:=\left\{v\in\R^d: \bigg\|\sum_{i\geq j+1}
      v^i\bigg\|<\gamma\bigg\|\sum_{i\leq j} v^i\bigg\|\right\},
  \end{displaymath}
  and
  \begin{displaymath}
    K_{j,\gamma}:=\left\{v\in\R^d: \bigg\|\sum_{i\leq j}
      v^i\bigg\|<\gamma\bigg\|\sum_{i\geq j+1} v^i\bigg\|\right\}.
  \end{displaymath}

  For each $j\in\{1,\ldots,k-1\}$ and any $n$, since we are assuming
  the linear decomposition $\R^d=\bigoplus_i E^i$ is orthogonal, it
  holds
  \begin{displaymath}
    \begin{split}
      \norm{A_n \sum_{i\leq j}v^i}^2&=\sum_{i\leq j}\norm{A_nv^i}^2
      \geq \left(e^{\lambda_j'-\delta/2}\right)^2\norm{\sum_{i\leq j}v^i}^2 \\
      &> (e^{\lambda_{j+1}'+\kappa})^2
      \norm{\frac{1}{\gamma}\sum_{i\geq j+1}v^i}^2 >
      \left(\frac{e^{\kappa-\delta/2}}{\gamma}\right)^2
      \norm{A_n\sum_{i\geq j+1} v^i}^2,
    \end{split}
  \end{displaymath}
  for every $v\in K_\gamma^j$. Analogously, for every
  $v\in K_{j,\gamma}$ we get
  \begin{displaymath}
    \begin{split}
      \norm{A_n^{-1}\sum_{i\leq j}v^i}^2&=\sum_{i\leq
        j}\norm{A_n^{-1}v^i}^2
      \leq \left(e^{\lambda_j'-\delta/3}\right)^{-2}\norm{\sum_{i\leq j}v^i}^2 \\
      &< (e^{\lambda_{j+1}'+\kappa+\delta/12})^{-2}
      \norm{\gamma\sum_{i\geq j+1}v^i}^2 <
      \left(\frac{\gamma}{e^{\kappa-\delta/4}}\right)^2
      \norm{A_n^{-1}\sum_{i\geq j+1} v^i}^2.
    \end{split}
  \end{displaymath}

  This implies that
  \begin{equation}
    \label{eq:cone-inclusions}
    A_n\big(K_\gamma^j\big)\subset K_{\gamma
      e^{-\kappa+\delta/4}}^j,\quad\text{and}\quad
    A_n^{-1}\big(K_{j,\gamma}\big)\subset K_{j,\gamma e^{-\kappa+\delta/4}},
  \end{equation}
  and for every $\gamma>0$, each $j\in\{1,\ldots,d-1\}$ and any
  $n\geq 1$.

  Then let us fix $\gamma>0$ sufficiently small such that
  \begin{equation}
    \label{eq:growth-in-cones}
    \norm{A_nv}\leq e^{\lambda_{j+1}'+\delta/3}\norm{v},
    \quad\text{and}\quad \norm{A_n^{-1} u}\geq
    e^{\lambda_{j}'-\delta/3}\norm{u}
  \end{equation}
  for every $j<k$, any $v\in K_{j,\gamma}$, all $u\in K^j_{\gamma}$
  and every $n\geq 1$.

  After fixing $\gamma$ as above, and taking into account
  \eqref{eq:cone-inclusions} and \eqref{eq:growth-in-cones}, we choose
  $\alpha_1>0$ small enough such that for any sequence of matrices
  $(C_n)_{n\geq 1}\subset\GL(d,\R)$ satisfying $\norm{A_n-C_n}<\alpha_1$
  for every $n$, it holds
  \begin{equation}
    \label{eq:Cn-conorm-estimate}
    \norm{C_n^{-1}}^{-1}\geq e^{\lambda'_k-\delta/2},
  \end{equation}
  \begin{equation}
    \label{eq:cone-inclusios-for-B}
    C_n\big(K_\gamma^j\big)\subset K_{\gamma e^{-\kappa+\delta/2}}^j,
    \quad\text{and}\quad C_n^{-1}\big(K_{j,\gamma}\big)\subset
    K_{j,\gamma e^{-\kappa+\delta/2}}.
  \end{equation}
  and
  \begin{equation}
    \label{eq:growth-in-cones-for-B}
    \norm{C_nv}\leq e^{\lambda_{j+1}'+\delta/2}\norm{v},
    \quad\text{and}\quad \norm{C_n^{-1} u}\geq
    e^{\lambda_{j}'-\delta/2}\norm{u},
  \end{equation}
  for every $j\in\{1,\ldots,d-1\}$, any $n\geq 1$, every
  $v\in K_{j,\gamma}$ and all $u\in K^j_\gamma$.

  Then, for any such sequence $(C_n)_{n\geq 1}$, let us define
  \begin{equation}
    \label{eq:Hk-def}
    H_k:=\bigcap_{n\geq 1}
    \big(C^{(n)}\big)^{-1}\big(\overline{K_{k-1,\gamma}}\big).
  \end{equation}

  By classical arguments one can easily show that $H_k$ is a linear
  subspace satisfying $\dim H_k=\dim E^k$. Putting together
  \eqref{eq:Cn-conorm-estimate}, \eqref{eq:growth-in-cones-for-B} and
  \eqref{eq:Hk-def} one can easily verify that
  \begin{displaymath}
    e^{n(\lambda_k -\delta/2)}\norm{v}\leq \norm{C^{(n)}v}\leq
    e^{n(\lambda_k'+\delta/2)}\norm{v}, \quad\forall v\in H^k.
  \end{displaymath}

  On the other hand, for every $v\in\R^d\setminus H^k$, there exists
  $N\in\N$ such that $C^{(n)}(v)\in K^{k-1}_\gamma$, and thus one gets
  \begin{displaymath}
    \liminf_{n\to+\infty} \frac{1}{n} \log\norm{C^{(n)}(v)} \geq
    \lambda'_{k-1}-\delta>\lambda'_k+\delta.
  \end{displaymath}

  Then, we define
  \begin{displaymath}
    H_{k-1}:=\bigcap_{n\geq 1} \big(C^{(n)}\big)^{-1}
    \big(\overline{K_{k-2,\gamma}}\big),
  \end{displaymath}
  and it can be shown that $H_{k-1}\supset H_k$ is a linear subspace
  satisfying $\dim H_{k-1}=\dim E^k+\dim E^{k-1}$, and by a similar
  argument to that used above, we get
  \begin{displaymath}
    e^{n(\lambda_{k-1}'-\delta/2)}\norm{v}\leq \norm{C^{(n)}(v)}\leq
    e^{n(\lambda_{k-1}'+\delta/2)}\norm{v}, \quad\forall v\in
    H_{k-1}\setminus H_k. 
  \end{displaymath}

  Inductively one defines the flag
  $\R^d=H_1\supset H_2\supset\cdots\supset H_k$ such that for any
  $v\in H_j\setminus H_{j+1}$, it holds
  \begin{equation}
    \label{eq:Cn-estimate-on-flag}
    \lambda_j'-\frac{\delta}{2}\leq \liminf_{n\to+\infty}\frac{1}{n}
    \log\norm{C^{(n)}(v)} \leq \limsup_{n\to+\infty}\frac{1}{n}
    \log\norm{C^{(n)}(v)}\leq \lambda_j'+\frac{\delta}{2}.
  \end{equation}

  Finally, by \eqref{eq:Ei-A-inv} and \eqref{eq:Ei-A-norm-estimate} we
  know that for each $i\in\{1,\ldots,d\}$, there exists
  $j\in\{1,\ldots,k\}$ such that
  \begin{displaymath}
    \abs{\frac{1}{n}\sigma_i\big(A^{(n)}\big)-\lambda_j'}\leq
    \frac{\delta}{2}, 
  \end{displaymath}
  and by \eqref{eq:Cn-estimate-on-flag} we also know that
  \begin{displaymath}
    \abs{\frac{1}{n}\sigma_i\big(C^{(n)}\big)-\lambda_j'}\leq
    \frac{\delta}{2}.
  \end{displaymath}
  Then, \eqref{eq:sigmai-An-vs-Cn} is consequence of last two
  estimates and the lemma is proved.
\end{proof}

\section{Fiber-wise Pesin Theory and Fake Invariant Sets}
\label{fakesets}

The purpose of this section consists introducing the so called
\emph{fake invariant sets,} which are something like a ``finite-time
stable and unstable sets'' and are mainly inspired by fake invariant
foliations of Burns and Wilkinson~\cite{BurnsWilkinsonErgoPartHypSys}
and pseudo-hyperbolic families of Hirsch, Pugh and
Shub~\cite{HirshPughShubInvMan}. We start stating a fiber-wise version
of some classical notions and results due to Pesin, and we shall just
sketch their proof.

Let $f\colon M\carr$ be a hyperbolic homeomorphism on the compact
metric space $(M,d_M)$, $A\colon M\to\Diff{}{1+\beta}(N)$ a
$C^\beta$-cocycle, for some $\beta>0$, and $F\colon M\times N\carr$ be
the induced skew-product given by
\eqref{eq:skew-product-diff-group}. Let $d_N$ be the distance function
on $N$ induced by its Riemannian structure. We shall consider the
space $M\times N$ endowed with the distance $d$ given by
\begin{displaymath}
  d(z,z'):=\sqrt{d_M(x,x')^2 + d_N(y,y')^2}, \quad\forall z=(x,y),\
  z'=(x',y')\in M\times N.
\end{displaymath}

Notice that, since $A$ is a $C^\beta$-H\"older map, if
$d_{C^{1+\beta}}$ denotes a distance on $\Diff{}{1+\beta}(N)$ inducing
the uniform $C^{1+\beta}$-topology, then there exists a real constant
$K>0$ such that
\begin{equation}
  \label{eq:A-Holder-constant}
  d_{C^{1+\beta}}\big(A(x),A(x')\big)\leq K d_M(x,x')^\beta,\quad \forall
  x,x'\in M,
\end{equation}
where we use the distance $d_{C^{1+\beta}}$ defined in Section~\ref{differentialtopology}.

\subsection{Fiber-wise Pesin charts}
\label{sec:fiber-wise-pesin-charts}

Consider the linear cocycle $\partial F\colon M\times TN\carr$ given
by \eqref{eq:derivative-F-cocycle}, which is in fact a fiber bundle
map over $F$. Let us fix an arbitrary ergodic $F$-invariant
probability measure $\mu$.

Then, there exists a finite collection of open subsets
$\Delta_1,\Delta_2,\ldots,\Delta_{k_0}$ of $N$ such that
\begin{displaymath}
  \mu\bigg(M\times (N\setminus\bigcup_{i=1}^{k_0}\Delta_i )\bigg)=0,
\end{displaymath}
and the tangent bundle $TN\big|_{\Delta_i}$ is trivial, \ie is
diffeomorphic to $\Delta_i\times\R^q$, for every
$i\in\{1,\ldots,k_0\}$, where $q$ denotes the dimension of $N$.  Via a
smooth trivialization of the tangent bundle $TN$ over each $\Delta_i$,
we can identify, up to $\mu$-measure zero, the derivative bundle map
$\partial F\colon M\times TN\carr$ with a linear cocycle over
$F\colon M\times N\carr$ (for details see \cite[Supplement, \S
2.b]{KatokHasselblatt}).

In other words, from now on we shall simply assume that
the bundle map $\partial F$ is represented by a linear cocycle
\begin{equation}  \label{eq:B-derivative-cocycle-def}
  B(z)
  := 
  DA(x)(y) \in\GL(q,\Bbb R), \quad\forall
  z=(x,y)\in M\times N,
\end{equation}
which is well-defined $\mu$-almost everywhere.

Now, let $r>0$ be the injectivity radius of the Rimeannian manifold
$N$. For any linear cocycle $L\colon M\times N\to\GL(q,\Bbb R)$, let
us define the injective map $\Phi_z\colon B(y,r)\to \Bbb R^q$ by
\begin{equation}\label{phiz}
  \Phi_z(y'):= L(z) \circ \exp_{y}^{-1}(y'),
  \quad\forall y'\in B(y,r),
\end{equation}
where $B(y,r)\subset \{x\}\times N $ is the ball of radius $r$ and
center $y$ with respect to the distance $d_N$.

For $r'>0$ sufficiently small (and smaller than $r$ in the previous
paragraph), from now on, $z=(x,y)$ and $z'=(x',y)$ will denote points of
$M\times N$, where $x' \in B_M(x,r')$. Now, we are interested in the cocycle $L=B$, the derivative 
cocycle~(\ref{eq:B-derivative-cocycle-def}), 
$\Phi_z$ will be the corresponding function defined in ~(\ref{phiz}),
 and hence, we can define the map
$h_{z,z'} \colon B(0,r')\subset\R^q\to\R^q$ by
\begin{equation}
  \label{eq:hzz'}
  h_{z,z'}(v):= \Phi_{F(z)}\circ \mathrm{pr}_2\circ
  F\big(x',\Phi_z^{-1}(v)\big), \quad\forall v\in B(0,r').
\end{equation}

Let $\lambda_1'>\cdots >\lambda_k'$ be the Lyapunov exponents of the
linear cocycle $B$ with respect to $\mu$, where each exponent
$\lambda_j'$ has multiplicity $n_j$, $j=1,\ldots k$. The following
result is a slight generalization of the classical construction of
Pesin charts:

\begin{theorem}
  \label{generalizedPesin}
  For every $\eta>0$, there exist a measurable set
  $\Lambda_\eta \subset M\times N$ with full $\mu$-measure and a
  tempered map $C_\eta\colon M\times N \to\GL(q,\R)$, satisfying the
  following properties:
  \begin{enumerate}
  \item For any $z\in \Lambda_\eta$, the matrix
    $B_\eta(z)=C_\eta(F(z))B(z)C_\eta^{-1}(z)$ has the Lyapunov block
    form
    \begin{displaymath}
      \mathrm{diag}\big(B_{\eta,1}(z),\cdots,B_{\eta,k}(z)\big),
    \end{displaymath}
    where each $B_{\eta,j}(z)$ is an $n_j\times n_j$ square matrix,
    $n_j$ denotes the multiplicity of the corresponding exponent
    $\lambda_j'$ and the norm $\|B_{\eta,j}(z)\|$ and co-norm
    $\|B_{\eta,j}(z)^{-1}\|^{-1}$ lie in the interval
    $[e^{\lambda_j'-\eta},e^{\lambda_j'+\eta}]$.

  \item There exists a measurable function
    $r_1\colon\Lambda_\eta \to (0, r']$, with
    \begin{equation}\label{r1size}
      e^{-\eta/\beta^2} <\frac{r_1(z)}{r_1(F(z))}<e^{\eta/\beta^2},
      \quad\forall z\in\Lambda_\eta,
    \end{equation}
    such that, for any point $z=(x,y)\in \Lambda_\eta$ and any
    $x'\in B_M(x, {r_1(z)})$, the map $h_{z',z}$ given by \eqref{eq:hzz'}
    where $z'=(x',y)$, satisfies
    \begin{equation}
      \label{hzz}
      d_{C^1}\big(h_{z,z'}, h_{z,z} \big) \leq \eta.
    \end{equation}

  \item There exist measurable functions
    $K\colon\Lambda_\eta \to(0,\infty)$ and
    $r_2\colon\Lambda_\eta\to (0,r']$ with
    \begin{align*}
      e^{-\eta} &< \frac {K(z)}{K(F(z))} < e^\eta, \\
      e^{-\eta/\beta} &<\frac{r_2(z)}{r_2(F(z))}<e^{\eta/\beta},
    \end{align*}
    such that, for every $z\in \Lambda_\eta$, and for any
    $y',y''\in \Phi_{z}^{-1}\Big(B(0,r_2(z))\Big)$,
    \begin{equation}
      \label{metricchange}
      \frac 12 d_N(y',y'')\leq \|\Phi_{z}(y')-\Phi_{z}(y'')\|\leq
      K(z)d_N(y',y''). 
    \end{equation}
  \end{enumerate}
\end{theorem}

\begin{proof}
  This result can be easily proved following \emph{mutatis mutandis}
  the proof of \cite[Theorem S.3.1]{KatokHasselblatt}. The only
  significant difference is that to control our regular neighborhood
  we need two different radii, namely $r_1$ and $r_2$, for the size of
  the balls on the base space and on the fibers.
\end{proof}

Then we recall some classical definitions of Pesin theory. Continuing
with the notations of Theorem~\ref{generalizedPesin}, every point
$z=(x,y)$ of $\Lambda_\eta$ is called a \emph{regular point,} and each
set $B_M(x,r_1(z))\times \Phi_z^{-1}(B(0,r_2(z)))$ is called \emph{a
  regular neighborhood} of $z$.

On the other hand, for each $\eta>0$ sufficiently small and any
$\varepsilon>0$, by Luzin's theorem, there exist $\ell>0$ and a
compact subset $\Lambda_{\eta,\ell}\subset\Lambda$ such that
$\mu(\Lambda_{\eta,\ell})>1-\varepsilon$, the functions
$r_1\big|_{\Lambda_{\eta,\ell}}$, $r_2\big|_{\Lambda_{\eta,\ell}}$,
$K\big|_{\Lambda_{\eta,\ell}}$ and $C_\eta\big|_{\Lambda_{\eta,\ell}}$
are continuous, and $\abs{r_1(z)^{-1}}\leq\ell$ and
$\abs{K(z)}\leq\ell$, for any $z\in\Lambda_{\eta,\ell}$.  The compact
set $\Lambda_{\eta,\ell}$ is called the \emph{Pesin uniformity block}
of tolerance $\eta$ and bound $\ell$.

\subsection{Fake Invariant Sets}
For the sake of simplicity of the notation, given an arbitrary point
$z_0=(x_0,y_0)\in\ M\times N$, we write $z_n=(x_n,y_n)=F^n(z_0)$, for
every $n\in\Z$.  Given any positive number $r$ we write
\begin{displaymath}
  U=U(z_n,r):= B_M(x_n,r)\times B_N(y_n,r),   
\end{displaymath}

Continuing with the notation we introduced in
\S\ref{sec:fiber-wise-pesin-charts}, let $\mu$ denote an ergodic
$F$-invariant measure, $\eta$ be a sufficiently small positive number
and $\ell$ a sufficiently large positive number, such that the Pesin
uniformity block $\Lambda_{\eta,\ell}\subset\Lambda_\eta$ has positive
$\mu$-measure. Hence by Poincar\'e recurrence theorem, for
$\mu$-almost every point $z_0\in\Lambda_{\eta,\ell}$, there exists a
positive integer number $N_0$ such that
$z_{N_0}\in\Lambda_{\eta,\ell}$.

Now we can state the main result of this section:

\begin{proposition}[Existence of fake invariant sets]
  \label{fakestablemanifold}
  There exist positive real constants
  $r^{(0)},\kappa, C, \widetilde C$ such that, given any pair of
  points $z_0,z_{N_0}\in\Lambda_{\eta,\ell}$ (for some $N_0>1$), there
  exists a partition $\widehat{\mathcal{U}}^{n,s}$ of
  $U(z_n,r^{(n)})$, for each $n\in\{0,1,\ldots,N_0\}$, where
  \begin{equation}
    \label{rndefi}
    r^{(n)}:=r^{(0)} e^{-\frac {\eta}{\beta^2} \min \{n,N_0-n \}},
  \end{equation}
  satisfying the following properties:

  \begin{enumerate}
  \item{\textbf{(Covering $f$-local stable sets)}} for any
    $z=(x,y) \in U(z_n,r^{(n)})$, if $\widehat{\mathcal{U}}^{n,s}(z)$
    denotes the atom of $\widehat{\mathcal{U}}^{n,s}$ containing $z$,
    then
    \begin{displaymath}
      \mathrm{pr}_1\Big(\widehat{\mathcal U}^{n,s}(z)\Big) =
      W^s_{2r^{(n)}}(x)\cap B_M(x_n,r^{(n)}). 
    \end{displaymath}
  \item{\textbf{(Local invariance)}} for every
    $n\in\{0,\ldots,N_0-1\}$ and any
    $z\in U\big(z_n, \frac{r^{(n)}}{C}\big)$,
    \begin{displaymath}
      F(z) \in U\big(z_{n+1},r^{(n+1)}\big), 
    \end{displaymath}
    and
    \begin{displaymath}
      F\Big(\widehat{\mathcal{U}}^{n,s}(z)\Big)\cap
      U\big(z_{n+1},r^{(n+1)}\big) \subset
      \widehat{\mathcal{U}}^{n+1,s}\big(F(z)\big). 
    \end{displaymath}
  \item{\textbf{(Uniform contraction)}} for every
    $k\in \{1,\cdots,N_0\}$ and any pair of points
    $z',z''\in \widehat {\mathcal U}^{0,s}(z_0)\cap
    B_M\Big(x_0,\frac{r^{(0)}}{C\widetilde C}\Big) \times
    B_N\Big(y_0,\frac{r^{(0)}}{C\widetilde C}\Big)$,
    \begin{equation}
      \label{contractioninequ}
      d\big(F^k (z'), F^k(z'')\big) \leq \widetilde C e^{-k\kappa}
      d(z',z'').
    \end{equation}
  \end{enumerate}
\end{proposition}

The partitions $\widehat{\mathcal U}^{n,s}$ given by
Proposition~\ref{fakestablemanifold} will be called a \emph{fake
  locally stable foliations} and their elements \emph{fake locally
  stable sets.} Similarly, we can define the \emph{fake locally
  unstable foliations} exhibiting analogous properties.

The rest of this section is devoted to the proof of
Proposition~\ref{fakestablemanifold}. First we need some lemmas:

\begin{lemma}\label{C1perturbation}
  Let $\R^q=\R^u\oplus\R^c\oplus\R^s$ be an orthogonal splitting and
  $L_i\colon\R^q\carr$, with $i\in\{1,\ldots,n\}$, be linear maps that
  preserve the splitting. Suppose there exists $\lambda\in(0,1)$ such
  that it holds
  \begin{align*}
    \max&\left\{\|L_i|_{\R^s}\|, \|(L_i|_{\Bbb R^u})^{-1}\|\right\} < \lambda;\\ 
    \max&\left\{\|L_i|_{\R^c}\|, \|(L_i|_{\Bbb R^c})^{-1}\|\right\}<
          \lambda^{-\frac 12},
  \end{align*}
  for each $i\in\{1,\ldots,n\}$.

  Then there exists a number $\alpha_2=\alpha_2(\lambda)>0$ such that if
  $f_i\colon\R^q\carr$ are $C^1$-diffeomorphisms satisfying
  $d_{C^1}(f_i,L_i)<\alpha_2$ for every $i\in\{1,\ldots,n\}$, then there
  exists a family of foliations $\mc{W}_i^s$ of $\R^q$, called
  \emph{stable foliations}, satisfying the following properties:
  \begin{enumerate}
  \item $f_i(\mathcal W_i^s)=\mathcal W_{i+1}^s$, for each
    $i\in\{1,\ldots,n-1\}$;
  \item for any $x\in\R^q$, any $i\in\{1,\ldots,n-1\}$ and every
    $y \in \mathcal W_i^s(x)$,
    \begin{displaymath}
      \norm{f_i^k(x)-f_i^k(y)}<\lambda^k \norm{x-y}, \quad \forall
      k\in\{1,\ldots,n-i\}. 
    \end{displaymath}
  \end{enumerate}

  An analogous result holds for unstable foliations.
\end{lemma}

\begin{proof}
  This is essentially Hadamard-Perron's theorem. See \cite[Theorem
  6.2.8]{KatokHasselblatt} for details.
\end{proof}

\begin{lemma}
  \label{3lestimate}
  Let $\delta$, $\eta$ and $\ell$ be positive numbers and suppose that
  $e^{\eta}<\frac 43$. Let $f,g\colon\R^q\carr$ be
  $C^1$-diffeomorphisms such that $d_{C^1}(f, g)<\delta$.  Let
  $A,B\in\GL(q,\R)$ be two linear transformations such that either
  $\|A\|<\ell e^{\eta}$ and $\|B\|<\ell$, or $\|A\|<\ell$ and
  $\|B\|<\ell e^{\eta}$. Then,
  \begin{equation}
    d_{C^1}(AfB,AgB)<4\ell^2 \delta.
  \end{equation}
\end{lemma}

\begin{proof}
  The estimate easily follows from the following computation:
  \begin{displaymath}
    \begin{split}
      & d_{C^1}(AfB(\cdot),AgB(\cdot)) \\
      = & d_{C^0}(AfB(\cdot),AgB(\cdot)) + d_{C^0} \big(A \circ
      Df(B(\cdot)) \circ B(\cdot),
      A  \circ   Dg(B(\cdot)) \circ B(\cdot)\big) \\
      \leq & \|A\| d_{C^0}(f,g) + \|A\| d_{C^0}(Df,Dg) \|B\| +
      \|A\| d_{C^0}(f,g)\|B\| \\
      \leq & 3\ell^2 e^{\eta} \delta< 4\ell^2 \delta.
    \end{split}
  \end{displaymath}
\end{proof}

For any small number $r>0$, we can choose a $C^1$ map
$\rho_r: \Bbb R^q\to [0,1]$ such that $\rho_r=1$ on $B(0,r)$ and
$\rho_r=0$ on $\R^q\setminus B(0,2r)$, and $\|\rho_r\|_{C^0}\leq 1$,
$\|D\rho_r\| = O(\frac 1r)$.  Then, consider any $C^{1+\beta}$
function $f : N\carr$, where $N$ is a closed smooth Riemannian
manifold.  Define a diffeomorphism $f^r: TN\to TN$ as follows,
\begin{equation}
  f^r(x,v)= \rho_r(v) \times \exp_{f(x)}^{-1}\circ f\circ \exp_x (v)+(1-\rho_r(v))\times Df_x(v).
\end{equation} 
Note that
\begin{equation}\label{f_rfunction}
  f^{r}(x,v):=\left \{
    \begin{array}{ll}
      \exp_{f(x)}^{-1}\circ f\circ \exp_x (v), & \text{ if } \|v\| \leq r,\\
      Df_x(v), & \text{ if } \|v\| \geq 2r.
    \end{array} \right. 
\end{equation}

Then we have the following $C^1$ estimate for this diffeomorphism.

\begin{lemma} \label{flocal}
  \begin{equation}
    d_{C^1}(f^r(x,\cdot),Df_x(\cdot)) =O(r^\beta).
  \end{equation}
\end{lemma}
\begin{proof}
  By definition,
  \begin{equation}
    f^r(x,v)-Df_x(v)=\rho_r(v) \times ( \exp_{f(x)}^{-1}\circ f\circ \exp_x (v) - Df_x(v)).
  \end{equation}
  Then, for $r$ small,
  \begin{align}
    & \|f^r(x,\cdot)-Df_x(\cdot)\|_{C^1} \\
    \leq  & \|\rho_r(\cdot) \times ( \exp_{f(x)}^{-1}\circ f\circ \exp_x (\cdot) - Df_x(\cdot))\|_{C^0} \nonumber\\
    + &\|D\rho_r(\cdot) \times  ( \exp_{f(x)}^{-1}\circ f\circ \exp_x (\cdot) - Df_x(\cdot))\|_{C^0}  \nonumber\\
    + & \|\rho_r(\cdot)  \times D( \exp_{f(x)}^{-1}\circ f\circ \exp_x (\cdot) - Df_x(\cdot))\|_{C^0}\nonumber\\
    \leq & (2r)C(2r)^\beta+ \frac{C}{r} (2r)C'(2r)^\beta+ C'(2r)^\beta\nonumber\\
    \leq & C ( r^{1+\beta} +r^\beta +r^\beta)  < C r^{\beta}.\nonumber
  \end{align}
\end{proof}

\begin{proof}[Proof of Proposition~\ref{fakestablemanifold}] 
  Let $\lambda_1'>\cdots>\lambda_k'$ be the distinct fibered Lyapunov
  exponents of the linear cocycle $B$ defined in
  (\ref{eq:B-derivative-cocycle-def}), with multiplicities
  $n_1,\cdots,n_k$, respectively.

  Fix a number
  \begin{equation}\label{kappa}
    \kappa<\frac 12 \min_{i} \{ \lambda_i'-\lambda_{i+1}', \tau \},
  \end{equation}
  where $\tau$ is the hyperbolicity rate of $f: M\to M$.  With respect
  to the number $e^{-\kappa}\in(0,1)$, 
  Lemma~\ref{C1perturbation} gives the constant
  $\alpha_2$.  Then we choose the number $\eta$ such that
  \begin{equation}\label{eta}
    \eta< \min \left\{ \beta^2\kappa,\frac 12 \min_i \{\lambda_i'-\lambda_{i+1}'\}-
    \kappa,\frac{\alpha_2}{2}\right\}.
  \end{equation}  

  Invoking Theorem~\ref{generalizedPesin}, we know there is a
  measurable map $C_\eta \colon \Lambda_\eta \to\GL(q,\Bbb R)$, a
  function $K\colon \Lambda_\eta \to (0,+\infty)$, with
  $K(z)> \max\{\|C_\eta(z)\|,\|C_\eta^{-1}(z)\|\}$, such that
  \begin{equation}
    e^{-\eta}<\frac{K(z)}{K(F(z))}<e^{\eta}.
  \end{equation}

  For any small $r>0$, $z=(x,y)\in \Lambda_\eta$, we can define
  $F^{r}_z$ as follows
  \begin{align}\label{F_rdefi}
    F^{r}_z(x',v)=
    \left \{
    \begin{array}{ll}
      \exp_{\text{pr}_2\circ F(z)}^{-1}  \circ F(x',\exp_y(v)), & \text{ if } \|v\|<r,\\
      D( A(x') ) (y) (v), & \text{ if } \|v\|>2r.
    \end{array} \right. 
  \end{align}
  For sufficiently $r>0$ sufficiently small $B(x,r) \subset M$,
  $F^{r}_z$ is well-defined on $B(x,r)\times \Bbb R^q$.  Moreover, by
  Lemma~\ref{flocal},
  for some constant $K$,
  \begin{equation}\label{K''estimate}
    \|F^{r}_z(x,\cdot)- B(z)(\cdot) \|_{C^1} \leq K r^\beta.
  \end{equation}

  Note that for some $\ell$ sufficiently large, the uniformity block
  $\Lambda_{\eta,\ell}\subset \Lambda_\eta$ has positive
  $\mu$-measure.  For any orbit segment $\{z_0,\cdots,z_{N_0}\}$, with
  $z_0,z_{N_0}\in \Lambda_{\eta,\ell}$, Theorem~\ref{generalizedPesin}
  item (1) gives $N_0$ linear maps
  \begin{equation}
    B_\eta(z_0), B_\eta(z_1),\cdots, B_\eta(z_{N_0-1}),
  \end{equation}
  all with Lyapunov block forms.

  Now we take
  \begin{equation}\label{r^0}
    r^{(0)}<\min \left\{  \frac {1}{\ell^2},
      (\frac{\alpha_2}{8K\ell^2})^{1/\beta} \right\}. 
  \end{equation}
  Then by (\ref{K''estimate}),
  \begin{equation}
    \|F^{r^{(0)}}_{z_0}(x_0,\cdot)- B(z_0)(\cdot) \|_{C^1}<\alpha_2/ 8\ell^2,
  \end{equation}
  and by our choice of the uniformity block,
  \begin{align}
    & \max(\|C_\eta(z_0)\|,\|C_\eta(z_0)^{-1}\| ) \leq K(z_0)\leq \ell,\\
    & \max(\|C_\eta(z_1)\|,\|C_\eta(z_1)^{-1}\| ) \leq K(z_1)\leq \ell e^{\eta}.
  \end{align}

  By Lemma~\ref{3lestimate}, we get 
  \begin{equation}\label{deltaestimate}
    \|C_\eta(z_1)\circ F_{z_0}^{r^{(0)}}(x_0,\cdot)\circ C_\eta^{-1}(z_0)- 
    B_\eta (z_0) \|_{C^1} <  \alpha_2/2. 
  \end{equation} 
  So, it follows from Theorem~\ref{generalizedPesin} item (2) that,
  for any $x'\in B(x_0,r^{(0)})$, it holds 
  \begin{equation}
    \|C_\eta(z_1)\circ F_{z_0}^{r^{(0)}}(x',\cdot)\circ C_\eta^{-1}(z_0)- 
    B_\eta (z_0) \|_{C^1} <  \alpha_2. 
  \end{equation} 

  Now, let us consider the map $F^{r^{(n)}}_{z_n}(x_n,\cdot)$, for
  $n\in\{1,\ldots,N_0-1\}$.

 By definition of  $r^{(n)}$ in~(\ref{rndefi}), and noting (\ref{r1size}), it  follows that for any $n=0,\cdots, N_0$, $U(z_n,r^{(n)})$ is contained in the regular  neighborhood $B_M(x_n,r_1(z))\times \Phi_{z_n}^{-1}(B(0,r_2(z)))$. Moreover, 
   \begin{align}
    r^{(n)}      
         <  &r^{(0)} e^{-\frac \eta\beta\min \{n,N_0-n\}} \\
         <  &(\frac{\alpha_2}{8K\ell^2})^{\frac 1\beta}e^{-\frac \eta\beta\min\{n,N_0-n\}} 
         =  (\frac{\alpha_2}{8K\ell^2 e^{\eta \min \{n,N_0-n\}}})^{\frac 1\beta}. \nonumber
  \end{align}
  By Lemma~\ref{flocal},
    \begin{equation}
    \|F^{r^{(n-1)}}_{z_{n-1}}(x_{n-1},\cdot)- B(z_{n-1})(\cdot) \|_{C^1}<\frac{\alpha_2}{8\ell^2} e^{-\eta \min \{n,N_0-n\}}.
  \end{equation}
  Noting that $K(z_{n-1}) \leq \ell e^{\eta \min \{n,N_0-n\}} $ and 
  $K(z_n)\leq \ell e^{\eta\min \{n+1,N_0-n-1\}}$, and applying Lemma~\ref{3lestimate}, we obtain
   \begin{align}\label{deltaestimate2}
    \|C_\eta(z_n)\circ F^{r^{(n-1)}}_{z_{n-1}}(x_{n-1},\cdot)\circ C_\eta^{-1}(z_{n-1}) - 
    B_\eta(z_{n-1}) \|_{C^1}<\frac{\alpha_2}{2}, 
  \end{align}

  Then it follows from Theorem~\ref{generalizedPesin} that for any
  $x'\in B(x_{n-1},r^{(n-1)})$,
  \begin{equation}
    \|C_\eta(z_n)\circ F_{z_{n-1}}^{r^{(n-1)}}(x',\cdot)\circ C_\eta^{-1}(z_{n-1})- 
    B_\eta (z_{n-1}) \|_{C^1} <  \alpha_2. 
  \end{equation}

  Now, by Lemma~\ref{C1perturbation} and the above estimates, the
  mappings
  $\{C_\eta(z_n)\circ F^{r^{(n-1)}}_{z_{n-1}}(x',\cdot)\circ
  C_\eta^{-1}(z_{n-1})\}_{n=1}^{N_0}$ admit stable foliations of
  $\Bbb R^d$.  We write these stable foliations as
  $\mathcal W^{n-1,s,x'}$.  
  
  The choice of $\eta$
  in (\ref{eta}) gives $2\kappa+2\eta<\min_i \{ \lambda_i'-\lambda_{i+1}' \}$. 
    For any $w\in \mathcal W^{n,s,x_n'}(v)$, where
  $x_n'\in B(x_n,r^{(n)})$,  the forward iterates of $x_n'$ will be denoted by 
  $x_{n+1}', \cdots,x_j'$. For $N_0\geq j\geq n$,
  \begin{align}\label{Lyapunovcontraction}
    d_E\Big(  & C_\eta(z_{j+1})\circ F^{r^{(j)}}_{z_j}(x_j',\cdot) \circ \cdots \circ 
                F^{r^{(n)}}_{z_n}(x_n',\cdot) \circ C_\eta(z_n)^{-1}(w), \\
              & C_\eta(z_{j+1})\circ F^{r^{(j)}}_{z_j}(x_j',\cdot) \circ \cdots F^{r^{(n)}}_{z_n}(x_n',\cdot) \circ C_\eta(z_n)^{-1}(v) \Big ) \nonumber \\
              &  <  e^{-2(j-n)\kappa} d_E(w,v), \nonumber
  \end{align}
  where we use $d_E$ to denote the standard Euclidean distance.

  We can now define the stable partition $\widehat{\mathcal U}^{n,s}$
  of $U(z_n,r^{(n)})$.  It is done in two steps.  Firstly, we consider
  the stable foliation $W^s$ of the ball $B(x_n,r^{(n)}) \subset  M$.
  Then, on each fiber over $x\in B(x_n,r^{(n)})$, the set
  $B(y_n,r^{(n)})$ is foliated by the the projection of the foliation
  $\mathcal W^{n,s,x}$ constructed above, i.e., the foliation
  $\exp_{y_n}\circ C_\eta(z_n)^{-1}(\mathcal W^{n,s,x})$ restricted to
  $B(y_n,r^{(n)})$.  Note item (1) of the proposition follows directly
  from this definition.

  For all $n=0,\cdots,N_0$,
  $r^{(n+1)}\geq r^{(n)} e^{-\frac {\eta}{\beta^2}}$, it is possible
  to choose a large constant $C$, such that
  \begin{equation}\label{localstableinvariance}
    F\big( U(z_n, \frac{r^{(n)}}{C} ) \big)  \subset  U(z_{n+1}, r^{(n+1)} ).
  \end{equation}
  This shows Item (2).
  
  Now we prove item (3). Define
  $\widetilde C =\max \{ \frac{\ell^2}{2},K_0 \}$, where $K_0$ is the
  constant in Definition~\ref{def:hyperbolic-homeo}.  Now we take
  points
  $z',z''\in \widehat{\mathcal U}^{0,s}(z_0)\cap
  B_M\Big(x_0,\frac{r^{(0)}}{C\widetilde C}\Big) \times
  B_N\Big(y_0,\frac{r^{(0)}}{C\widetilde C}\Big)$, and we will write
  $z'_k=(x_k',y_k')=F^k(z'), z''_k=(x''_k,y''_k)=F^k(z'')$ for all
  $k$.  Inductively, both points $z_k'$ and $z_k''$
  are contained in
  $\widehat{\mathcal U}^{k,s}(z_k)\cap
  B_M\Big(x_k,\frac{r^{(k)}}{C}\Big) \times
  B_N\Big(y_k,\frac{r^{(k)}}{C}\Big) $. Then item (2) ensures the
  estimate (\ref{Lyapunovcontraction}) can be applied at each step.

  Thus, by hyperbolicity of the base dynamics,
  \begin{equation}
    d_M(x_k', x_k'')\leq K_0 e^{-\tau k} d_M(x_0',x''_0) \leq \frac{r^{(k)}}{C}.
  \end{equation}
  The second coordinate has a similar estimate. Indeed, it follows
  from~(\ref{Lyapunovcontraction}) that, for any $k=1,\cdots,N_0-1$,
  \begin{equation}\label{finalcontraction}
    d(y_k',y_k'') \leq d(y_k',y_k) +d(y_k'',y_k).
  \end{equation}
  We estimate
  \begin{align}
    d(y_k',y_k)\leq   & \ell e^{k\eta} d_E \{ C_\eta(z_{k})\circ F_{z_{k-1}}^{r^{(k-1)}}(x_{k-1}',\cdot)\circ \cdots 
                        \circ F_{z_0}^{r^{(0)}}\big (x_0', \exp_{y_0}^{-1}(y'_0) \big), \nonumber\\
                      & C_\eta(z_{k})\circ F_{z_{k-1}}^{r^{(k-1)}}(x_{k-1},\cdot)\circ \cdots \circ
                        F_{z_0}^{r^{(0)}} (x_0,0) \} \nonumber\\
    \leq   & \ell e^{-k\kappa} d(C_\eta(z_0) \circ \exp_{y_0}^{-1}(y'_0), 0) \nonumber\\
    \leq  &  \ell^2 e^{-k\kappa}d(y_0',y_0)   \leq  \frac{r^{(k)}}{2C}. \nonumber
  \end{align}
  The term $d(y_k'',y_k)$ satisfies a similar estimate. So we have
  $d(y_k',y_k'')\leq \frac{r^{(k)}}{C}$. Combining the first and the
  second coordinate estimates, we have finished the proof of item (3)
  of the proposition.
\end{proof}

\section{Fiber closing lemma and proof of main results}
In this section we prove Theorem~\ref{thm:main-thm}. We start
with a technical result which is some kind of ``Closing Lemma along
the fibers''.

Let $F\colon M\times N \carr$ be the skew-product induced by the
hyperbolic homeomorphism $f\colon M\carr$ and the $C^\beta$-cocycle
$A\colon M\to \text{Diff}^{1+\beta}(N)$. Let $\mu$ be an ergodic
$F$-invariant probability measure. Let $\eta,\ell$ be positive numbers
with $\ell$ large enough in order to guarantee that the Pesin
uniformity block $\Lambda_{\eta,\ell}$ of $F$ has positive
$\mu$-measure. Let $r^{(0)}, C,\widetilde C, \kappa$ be the constants
given by Proposition~\ref{fakestablemanifold}.  In what follows, when
there is no confusion, we denote the iterates of a point
$z_0=(x_0,y_0)\in M\times N$ by $z_i=(x_i,y_i)=F^i(z_0)$, for every
$i\in\Z$. Since $\mu(\Lambda_{\eta,\ell})>0$, notice that by Poincaré
recurrence theorem, there exists $z_0\in\Lambda_{\eta,\ell}$ such that
$z_n$ belongs to $\Lambda_{\eta,\ell}$ too, for some $n\geq 1$. 

\begin{lemma}
  \label{fiber-shadowing}
  There exist constants $\varepsilon_0, K>0$ such that for any point
  $z_0\in\Lambda_{\eta,\ell}$ such that there is $n\geq 1$ with
  $z_n\in\Lambda_{\eta,\ell}$ and $d(z_0,z_n)<\varepsilon_0$, there is
  some $F$-orbit segment $\{z_0'',z_1'',\cdots,z_n''\}$, with
  $\text{pr}_1(z_0'')=\text{pr}_1(z_n'')$ and satisfying
  \begin{equation}
    d(z_i'',z_i)\leq K d(z_0,z_n) e^{-\kappa \min\{i,n-i\}}.
  \end{equation}
\end{lemma}

\begin{proof}
  Let us define $K=\max \{ \ell, 2C\widetilde C \}$ and 
  $\varepsilon_0 = r^{(0)}/K$.

  First note that for any orbit segment satisfying
  $d(z_0,z_n)<\varepsilon_0$ and $z_0,z_n\in \Lambda_{\eta,\ell}$, the
  point $z_n$ lies in the regular neighborhood of $z_0$.

  In particular, $d(x_0,x_n)<\varepsilon_0$ and hence we can apply
  Anosov Closing Lemma to conclude there is a unique periodic point
  $p$ with period $n$ which exponentially shadows the orbit segment
  $\{x_0,\cdots,x_n\}$.  Consider the point
  $x_0'=W^s(x_0)\cap W^u(p)$. Then, for each $i=0,\cdots,n$ it holds
  \begin{equation}
    d( x_i,x_i')<  \frac K2  e^{-\tau\min\{i,n-i\}} d(x_0,x_n).
  \end{equation} 

  Since $x_0'\in W^s_{\frac{r^{(0)}}{K}} (x_0)$, by
  Proposition~\ref{fakestablemanifold}, $x_0'$ has a pre-image by
  $\text{pr}_1$ which is contained in the local fake stable set of
  $z_0$.  Denote this point by
  $z_0'\in \widehat{\mathcal U}^{0,s} (z_0)$.  Moreover, by our choice
  of $K$ it holds
  \begin{equation}\label{ii'}
    d(z_i, z_i')<  \frac K2 d(z_0,z_0') e^{-\kappa \min\{i,n-i\}},
    \quad\forall i\in \{1,\cdots,n\} 
  \end{equation} 
  In particular, $d(z_n',z_n)< r^{(n)}$. Notice that
  $r^{(n)}=r^{(0)}$, too.

  Now let us consider the point $x_n'=\text{pr}_1(z_n') \in M$.  Since
  $ d(p,x_n') \leq d(p,x_n)+d(x_n,x_n')$, the periodic point
  $p\in W^u_{2\varepsilon_0}(x_n')$.  Applying
  Proposition~\ref{fakestablemanifold} in the backward direction, we
  conclude the point $p$ has a pre-image under $\text{pr}_1$, denoted
  as $z_n'' \in \widehat {\mathcal U}^{n,u} (z_n')$, such that if we
  write $z_i'': = F^{-(n-i)} (z_n'')$ for all $i=0,\cdots,n-1$, then
  it holds
  \begin{equation}\label{i'i''}
    d(z_i'',z_i') <  \frac K2 d(z_n'',z_n') e^{-\kappa \min\{i, n-i\}}. 
  \end{equation}

  Combing both ~(\ref{ii'}) and ~(\ref{i'i''}), for all
  $i=0,\cdots, n$,
  \begin{align}\label{w''wexponential}
    d(z_i'',z_i) \leq d(z_i'',z_i')+d(z_i',z_i) 
    \leq   
    K d(z_n,z_0) e^{-\kappa \min \{i,n-i \}},
  \end{align}
  where the last inequality is consequence of the fact that
  $d(z_0,z_n)\sim d(z_0,z_0')$ and $d(z_0,z_n)\sim d(z_n'',z_n')$,
  maybe increasing $K$ if it is necessary.  Finally it clearly holds
  $\text{pr}_1(z_n'')=\text{pr}_1(z_0'')=p$.
\end{proof}

\begin{proof}[Proof of Theorem~\ref{thm:main-thm}] 
  According to Theorem 3.1 of \cite{KocsardPotrieLivThm}, under the
  periodic orbit condition (\ref{eq:POC}), it suffices to show that
  the fiber-wise Lyapunov exponents all vanish, with respect to any
  ergodic $F$-invariant probability measure $\mu$.

  Reasoning by contradiction, suppose that there exists an ergodic
  $F$-invariant measure $\mu$ such that their fibered Lyapunov
  exponents, listed by $\lambda_1\geq \cdots\geq \lambda_q$, are not
  all equal to zero.  If the measure is supported on finitely many
  fibers, then we immediately get a contradiction, because the
  periodic orbit condition clearly implies all Lyapunov exponents
  vanish.  If this is not the case, we consider a positive number
  $\delta$ which is smaller than the absolute value of one
  non-vanishing Lyapunov exponent.

  We claim that there exists an orbit segment $z_0'',\cdots,z_n''$
  with $\text{pr}_1(z_0'')=\text{pr}_1(z_n'')$, such that, if we
  denote the singular values of $D A^{(n)} (x_0'') (y_0'')$ by
  $e^{n\sigma_1}\geq \cdots\geq e^{n\sigma_d}$, then they satisfy
  $|\sigma_j-\lambda_j|\leq \delta$ for all $1\leq j\leq d$. In such a
  case, this linear map has at least one non-vanishing singular value,
  contradicting the periodic orbit condition (\ref{eq:POC}).

  So, it just remains to prove our claim.  Let $\kappa$ be a positive
  real number satisfying estimate \eqref{kappa}, and let
  $\alpha_1=\alpha_1(\delta/2)$ be the positive constant given by
  Lemma~\ref{singularapproximating}. Then, we apply
  Lemma~\ref{C1perturbation} getting a new constant
  $\alpha_2:=\alpha(e^{-\kappa})$; and we take
  $\alpha:=\min\{\alpha_1,\alpha_2\}$. Thus, let $\eta$ be a positive
  number satisfying \eqref{eta}. By Theorem~\ref{generalizedPesin}, we
  can choose a uniformity block $\Lambda_{\eta,\ell}$ with positive
  $\mu$-measure.  Finally, let $\varepsilon_0$ and $K$ be the positive
  constants given by Lemma~\ref{fiber-shadowing}.

  Then, we choose a point $z$ in $\Lambda_{\eta,\ell}$ which is a
  $\mu$-density point of $\Lambda_{\eta,\ell}$ and consequently, the
  ball $B(z, \varepsilon_0/2)$ intersects $\Lambda_{\eta,\ell}$ with
  positive $\mu$-measure.  By Poincar\'e's recurrence theorem, we can
  assume there are infinitely many natural numbers $n$ such that the
  $F$-orbit of $z$ satisfies 
  $z_0,z_n\in B(z,\varepsilon_0/2) \cap \Lambda_{\eta,\ell}$.
  
  Applying Lemma~\ref{fiber-shadowing}, there is an $F$-orbit segment
  $z_0'',z_1'',\cdots,z_n''$, with
  $\text{pr}_1(z_0'')=\text{pr}_1(z_n'')$, such that,
  \begin{equation}
    d(z_i,z_i'')< K d(z_0,z_n) e^{-\kappa\min \{i,n-i\} } \text{ for all }i=0,\cdots,n.
  \end{equation}

  Thus, by formula \eqref{deltaestimate2} we know
  that 
  \begin{equation}
    \|C_\eta(z_i)\circ F_{z_{i-1}}^{r^{(i-1)}}(z_{i-1},\cdot)\circ C_\eta^{-1}(z_{i-1})^{-1}- 
    B_\eta(z_{i-1}) \|_{C^1} <\alpha. 
  \end{equation} 

  So, if
  $\tilde\sigma_1^{(n)}\geq \tilde\sigma_2^{(n)}\geq \ldots\geq \tilde
  \sigma_q^{(n)}$ denotes the singular values of the linear map
  $C_{\eta}(z_n) \circ DA^n(x_0'')(y_0'') \circ C_{\eta}(z_0)^{-1}$,
  then invoking 
  Lemma~\ref{singularapproximating} for the constant $\delta /2$, we
  can guarantee that
  \begin{displaymath}
    \abs{\frac{1}{n}\log\tilde\sigma_j^{(n)}-\lambda_j}
    <\frac{\delta}{2}, 
  \end{displaymath}
  for each $i\in\{1,\ldots,q\}$. 

  Finally, note that the changes of Lyapunov coordinates at $z_0$ and $z_n$, namely
  $C_\eta(z_0)$ and $C_\eta(z_n)$ and their inverses, are bounded by $\ell$. So we can 
  apply Lemma~\ref{approximating} to conclude
  that, if
  $\sigma_1^{(n)}\geq \sigma_2^{(n)}\geq \ldots\geq \sigma_q^{(n)}$
  are the singular values of the linear map $DA^n(x_0'')(y_0'')$, then
  it holds
  \begin{displaymath}
    \abs{\frac{1}{n}\log\tilde\sigma_j^{(n)}-\lambda_j}
    <\delta, 
  \end{displaymath}
  for every $i$, completing the proof of the theorem.
\end{proof}

\bibliographystyle{amsalpha} 
\bibliography{references}

\providecommand{\bysame}{\leavevmode\hbox to3em{\hrulefill}\thinspace}
\providecommand{\MR}{\relax\ifhmode\unskip\space\fi MR }
\providecommand{\MRhref}[2]{%
  \href{http://www.ams.org/mathscinet-getitem?mr=#1}{#2}
}
\providecommand{\href}[2]{#2}
\begin{thebibliography}{dlLW11}

\bibitem[BW10]{BurnsWilkinsonErgoPartHypSys}
K.~Burns and A.~Wilkinson, \emph{On the ergodicity of partially hyperbolic
  systems}, Annals of Mathematics (2010), 451--489.

\bibitem[dlLW10]{delaLlaveWindsorLivThmNonComm}
R.~de~la Llave and A.~Windsor, \emph{Liv\v sic theorems for non-commutative
  groups including diffeomorphism groups and results on the existence of
  conformal structures for anosov systems}, Ergodic Theory Dynam. Systems
  \textbf{30} (2010), no.~4, 1055--1100. \MR{2669410}

\bibitem[dlLW11]{delaLlaveWindsorSmoorhDependence}
\bysame, \emph{Smooth dependence on parameters of solutions to cohomology
  equations over anosov systems with applications to cohomology equations on
  diffeomorphism groups}, Discrete Contin. Dyn. Syst. \textbf{29} (2011),
  no.~3, 1141--1154. \MR{2773168}

\bibitem[Guy]{GuysinskyLivsic}
M.~Guysinsky, \emph{Liv{\v{s}}ic theorem for cocycles with values in the group
  of diffeomorphisms}, in preparation.

\bibitem[HPS77]{HirshPughShubInvMan}
M.~W. Hirsch, C.~C. Pugh, and M.~Shub, \emph{Invariant manifolds}, Lecture
  Notes in Mathematics, Vol. 583, Springer-Verlag, Berlin, 1977. \MR{0501173}

\bibitem[Kal11]{KalininLivThmMat}
B.~Kalinin, \emph{Liv\v sic theorem for matrix cocycles}, Ann. of Math.
  \textbf{173} (2011), no.~2, 1025--1042. \MR{2776369}

\bibitem[KH96]{KatokHasselblatt}
A.~Katok and B.~Hasselblatt, \emph{Introduction to the modern theory of
  dynamical systems}, vol.~54, Cambridge Univ. Pr., 1996.

\bibitem[KK96]{KatokKononenkoCocStab}
A.~Katok and A.~Kononenko, \emph{Cocycle stability for partially hyperbolic
  systems}, Math. Res. Letters \textbf{3} (1996), 191--210.

\bibitem[KP16]{KocsardPotrieLivThm}
A.~Kocsard and R.~Potrie, \emph{Liv\v{s}ic theorem for low-dimensional
  diffeomorphism cocycles}, Comment. Math. Helv. \textbf{91} (2016), 39--64.

\bibitem[Liv71]{LivCertPropHomolYSyst}
A.~N. Liv\v{s}ic, \emph{Certain properties of the homology of {Y-systems}},
  Mat. Zametki \textbf{10} (1971), 555--564.

\bibitem[Liv72a]{LivsicCohomDynSys}
\bysame, \emph{Cohomology of dynamical systems}, Izv. Akad. Nauk SSSR Ser. Mat.
  \textbf{36} (1972), 1296--1320.

\bibitem[Liv72b]{LivHomolDynSys}
\bysame, \emph{The homology of dynamical systems}, Uspehi Mat. Nauk \textbf{27}
  (1972), no.~3(165), 203--204. \MR{0394768}

\bibitem[NP13]{NavasPonceLivAnalGerms}
A.~Navas and M.~Ponce, \emph{A {Liv\v{s}ic} type theorem for germs of analytic
  diffeomorphisms}, Nonlinearity \textbf{26} (2013), no.~1, 297--305.

\bibitem[NT95]{NiticaTorokCohomolDynSyst}
V.~Ni\c{t}ic\u{a} and A.~T\"or\"ok, \emph{Cohomology of dynamical systems and
  rigidity of partially hyperbolic actions of higher-rank lattices}, Duke Math.
  J. \textbf{79} (1995), no.~3, 751--810.

\bibitem[NT96]{NiticaTorokRegularResultsSolLiv}
V.~Ni\c{t}ic\u{a} and A.~T\"{o}r\"{o}k, \emph{Regularity results for the
  solutions of the {Livsic} cohomology equation with values in diffeomorphism
  groups}, Ergodic Theory Dynam. Systems \textbf{16} (1996), no.~2, 325--333.

\bibitem[Wil13]{WilkinsonCohomoEq}
A.~Wilkinson, \emph{The cohomological equation for partially hyperbolic
  diffeomorphisms}, Ast{\'{e}}risque \textbf{358} (2013), 75--165. \MR{3203217}

\end{thebibliography}

\end{document}